\documentclass[final]{siamltex}
\usepackage{latexsym,amssymb,amsmath,amsfonts}
\usepackage{mathrsfs}
\usepackage{graphicx}
\newcommand{\R}{{\mathbb R}}

\newcommand{\N}{{\mathbb N}}
\newcommand{\C}{{\mathbb C}}

\newcommand{\be}{\begin{eqnarray}}
\newcommand{\ben}{\begin{eqnarray*}}
\newcommand{\en}{\end{eqnarray}}
\newcommand{\enn}{\end{eqnarray*}}
\newcommand{\ba}{\backslash}
\newcommand{\pa}{\partial}

\newcommand{\ov}{\overline}
\newcommand{\curl}{{\rm curl\,}}
\newcommand{\curlcurl}{{\rm curlcurl\,}}
\newcommand{\Div}{{\rm Div\,}}

\newcommand{\G}{\Gamma}
\newcommand{\eps}{\epsilon}

\newcommand{\Om}{\Omega}
\newcommand{\om}{\omega}
\newcommand{\sig}{\sigma}

\newcommand{\al}{\alpha}
\newcommand{\la}{\lambda}

\newcommand{\wi}{\widetilde}

\newtheorem{remark}[theorem]{Remark}

\begin{document}
\renewcommand{\theequation}{\arabic{section}.\arabic{equation}}
\title{\bf The inverse electromagnetic scattering problem in a piecewise homogeneous medium}
\author{Xiaodong Liu\thanks{LSEC and Institute of Applied Mathematics, AMSS,
Chinese Academy of Sciences,Beijing 100190, China ({\tt
lxd230@163.com}).}
\and Bo Zhang\thanks{LSEC and Institute of Applied
Mathematics, AMSS, Chinese Academy of Sciences, Beijing 100190,
China ({\tt b.zhang@amt.ac.cn}).}
\and Jiaqing Yang\thanks{LSEC and
Institute of Applied Mathematics, AMSS, Chinese Academy of Sciences,
Beijing 100190, China ({\tt jiaqingyang@amss.ac.cn}).}
}

\maketitle

\begin{abstract}
This paper is concerned with the problem of scattering of
time-harmonic electromagnetic waves from an impenetrable obstacle in
a piecewise homogeneous medium. The well-posedness of the direct
problem is established, employing the integral equation method.
Inspired by a novel idea developed by H\"{a}hner \cite{Ha}, we prove
that the penetrable interface between layers can be uniquely
determined from a knowledge of the electric far field pattern for
incident plane waves. Then, using the idea developed by Liu \&
Zhang \cite{LZip}, a new mixed reciprocity relation is obtained and
used to show that the impenetrable obstacle with its physical
property can also be recovered. Note that the wave numbers in the
corresponding medium may be different and therefore this work can be
considered as a generalization of the uniqueness result of \cite{LZaa}.

\vspace{.2in}


\end{abstract}

\pagestyle{myheadings}
\thispagestyle{plain}
 \markboth{X. LIU, B. ZHANG \& J. YANG}{Inverse scattering in a piecewise homogeneous medium}

\section{Introduction}
\setcounter{equation}{0}

We consider the scattering of time-harmonic electromagnetic plane
waves with frequency $\om>0$ by an impenetrable obstacle which is
embedded in a piecewise homogeneous medium. For simplicity, and
without loss of generality, in this paper we restrict ourself to the
case where the obstacle is buried in a two-layered piecewise
homogeneous medium, as shown in Figure \ref{fig1}.
Note that our method and results can be easily extended to the multi-layered case.
Precisely, let $\Om_{2}\subset\R^3$ denote the impenetrable obstacle which is an
open bounded region with a $C^{2}$ boundary $S_{1}$ and let
$\R^3\ba\ov{\Om_{2}}$ denote the the background medium which is
divided by means of a closed $C^{2,\al}\;(0<\al<1)$ surface $S_{0}$
into two connected domains $\Om_{0}$ and $\Om_{1}$. Let $\Om$ denote
the complement of $\Om_{0}$, that is, $\Om:=\R^{3}\ba\ov{\Om_0}$. We
assume that the boundary $S_{1}$ of the obstacle $\Om_{2}$ has a
dissection $S_{1}=\ov{\G}_{1}\cup\ov{\G}_{2}$, where $\G_{1}$ and
$\G_{2}$ are two disjoint, relatively open subsets of $S_{1}$.

The electromagnetic properties of the homogeneous medium in
$\Om_{0}$ are described by space independent electric permittivity
$\eps_{0}>0$, magnetic permeability $\mu_{0}>0$ and vanishing
electric conductivity $\sig_{0}=0$. The electromagnetic properties
of the homogeneous medium in $\Om_{1}$ are determined by space
independent electric permittivity $\eps_{1}>0$, magnetic
permeability $\mu_{1}>0$ and electric conductivity $\sig_{1}\geq0$.
We define the wave number $k_j$ in the corresponding medium
$\Om_{j}$ by $k_{j}^{2}=(\eps_{j}+i\sig_{j}/\om)\mu_{j}\om^{2}$ with
$\Re k_{j}>0,\Im k_{j}\geq0\; (j=0,1)$.

\begin{figure}[htbp]
\centering \scalebox{0.38}{\includegraphics{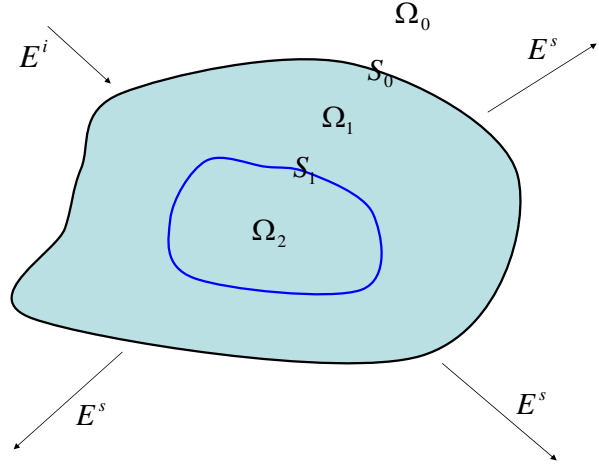}}
\caption{Scattering in a two-layered background medium}\label{fig1}
\end{figure}

For convenience we denote by $T(S_{j})$ and
$T^{0,\al}(S_{j})\;(j=0,1)$, $0<\al<1$, the spaces of all continuous
and uniformly H\"{o}lder continuous tangential fields equipped with
the supremum norm and the H\"{o}lder norm, respectively. Then we introduce the normed spaces
of tangential fields possessing a surface divergence (see section 6.3 in \cite{CK}) by
 \ben
 T_{d}(S_{j})&:=&\{a\in T(S_{j}): \Div a\in C(S_{j})\},\\
 T_{d}^{0,\al}(S_{j})&:=&\{a\in T^{0,\al}(S_{j}): \Div a\in C^{0,\al}(S_{j})\}.
 \enn
Let $T^{2}(S_{j})$ denote the completion of
$T(S_{j})$ with respect to the usual $L^{2}-$norm and
$T_{d}^{2}(S_{j})$ denote the tangential fields in $T^{2}(S_{j})$
with a surface divergence in $L^{2}(S_{j})\; (j=0,1)$. We denote by
$\nu(x)$ the unit normal vector to a surface at the point $x$. For a
closed surface it is directed into the exterior of the surface. For
two vectors $a, b \in \C^{3}$ we write $a\cdot b$ for the scalar
product and $a\times b$ for the vector product.

Given four tangential fields $T_{1}, T_{2}\in T_{d}^{0,\al}(S_{0}),
T_{3}\in T_{d}^{0,\al}(\G_{1})$ and $T_{4}\in T^{0,\al}(\G_{2})$,
the {\bf\em direct problem} consists in finding a solution $E,H\in
C^{1}(\Om_{0})\cap C(\ov{\Om_{0}}),\;F,G\in C^{1}(\Om_{1})\cap
C(\ov{\Om_{1}})$ to the Maxwell equations
 \be
 \label{0EH} \curl E-ik_{0}H=0,\qquad\curl H+ik_{0}E=0\qquad \mbox{in}\;\; \Om_{0},\\
 \label{0FG} \curl F-ik_{1}G=0,\qquad\curl G+ik_{1}F=0\qquad \mbox{in}\;\; \Om_{1},
 \en
which satisfies the Silver-M\"{u}ller radiation condition
 \be
 \label{0SMrc} \lim_{r\rightarrow\infty}(H\times x-r E)=0
 \en
where $r=|x|$ and the limit holds uniformly in all directions $x/r$,
the transmission boundary conditions
 \be
 \label{0tr}
\nu\times E-\la_{E}\nu\times F=T_{1},\qquad\nu\times
H-\la_{H}\nu\times G=T_{2} \qquad\mbox{on}\;\;S_0
 \en
with constants $\la_{E}$ and $\la_{H}$ given by
$\la_{E}=\sqrt{\eps_{0}/(\eps_{1}+i\sig_{1}/\om)}$,
$\la_{H}=\sqrt{\mu_{0}/\mu_{1}}$, and the boundary condition
 \ben
 \label{0bc}\mathscr{B}(F)=0\qquad \mbox{on}\; S_{1}
 \enn
with the operator $\mathscr{B}$ depending on the nature of the
obstacle $\Om_{2}$. Precisely, the boundary condition on $S_{1}$ is
understood as:
 \be
 \label{0pbc}\nu\times F=T_{3}  \qquad&&    \mbox{on}\; \G_{1},\\
 \label{0ibc}\nu\times G-\frac{\la}{k_{1}}(\nu\times F)\times\nu
                    =T_{4}\qquad &&\mbox{on}\; \G_{2}.
 \en
with a positive constant $\la$. Note that the case
$\G_{2}=\emptyset$ corresponds to a {\em perfectly conducting}
obstacle and the case $\G_{1}=\emptyset$ leads to an impedance
boundary condition corresponding to an obstacle which is not
perfectly conducting but does not allow the electromagnetic wave to
penetrate deeply into the obstacle. The scattering problems with
mixed boundary conditions widely occur in practical applications,
e.g., in the use of electromagnetic waves to detect "hostile"
objects where the boundary, or more generally a portion of the
boundary, is coated with an unknown material in order to avoid
detection. We refer to \cite{CC,CCM,CCM05} for the physical
relevance and practical implication of the electromagnetic
scattering by obstacles with a mixed boundary condition in a
homogeneous medium.

In the next section, an integral equation method is employed to
establish the well-posedness of the direct problem.
A mixed reciprocity relation will also be proved. These results will play
an important role in the proof of the uniqueness results in the
inverse problem.

The radiation condition (\ref{0SMrc}) ensures uniqueness of
solutions to the exterior boundary value problem and leads to an
asymptotic behavior of the form
 \be
 \label{0ab}
E(x)=\frac{e^{ik_0|x|}}{|x|}\left\{E^\infty(\widehat{x})+O(\frac{1}{|x|})\right\},
\qquad\mbox{as}\;\;|x|\rightarrow\infty
 \en
uniformly in all directions $\widehat{x}={x}/{|x|}$, where the
vector field $E^{\infty}$ defined on the unit sphere $S^{2}$ is
known as the electric far field pattern.

We consider the scattering of electromagnetic plane waves
 \ben\label{pw}
E^i(x,d,q)&=&\frac{i}{k_{0}}\curl\curl qe^{ik_{0}x\cdot d}
   =ik_{0}(d\times q)\times de^{ik_{0}x\cdot d},\\
H^i(x,d,q)&=&\curl qe^{ik_{0}x\cdot d}=ik_{0}d\times
   qe^{ik_{0}x\cdot d}
 \enn
where the unit vector $d$ describes the direction of propagation and
the constant vector $q$ gives the polarization. Then we have
$T_{1}=-\nu\times E^i(x,d,q), T_{2}=-\nu\times H^i(x,d,q)$ on
$S_{0}$, $T_{3}=0$ on $\G_{1}$ and $T_{4}=0$ on $\G_{2}$ in the
boundary problem (\ref{0EH})-(\ref{0ibc}). Throughout this paper, we
will indicate the dependence of the corresponding scattered field,
total field and far field pattern on the incident direction $d$ and
the polarization $q$ by writing $E^{s}(\cdot,d,q),
H^{s}(\cdot,d,q)$, $E(\cdot,d,q), H(\cdot,d,q)$ and
$E^{\infty}(\cdot,d,q), H^{\infty}(\cdot,d,q)$, respectively.

The {\bf\em inverse problem} we consider in this paper is, given the
wave numbers $k_{j}\;(j=0,1)$, the constants $\la_{E}$ and $\la_{H}$
and the electric far field pattern $E^{\infty}(\widehat{x},d,q)$ for
all observation directions $\widehat{x}\in S^{2}$, all incident
directions $d\in S^{2}$ and all polarizations $q\in \R^{3}$, to
determine the interface $S_{0}$ and the obstacle $\Om_{2}$ with its
physical property $\mathscr{B}$. As usual in most of the inverse
problems, the first question to ask in this context is the
identifiability, i.e. whether an obstacle can be identified from
knowledge of its far-field pattern. Mathematically, the
identifiability is the uniqueness issue which is of theoretical
interest and is required in order to proceed to efficient numerical
methods of solutions.

For scattering problems in a homogeneous medium, there has been
an extensive study in the literature; see, e.g., uniqueness results
for scattering from a perfect conductor by Colton \& Kress \cite{CK},
for scattering from an impenetrable obstacle with the boundary condition
(\ref{0ibc}) by Kress \cite{Kress}, for scattering from a penetrable
obstacle with transmission boundary conditions by H\"{a}hner \cite{Ha} and
for scattering from a penetrable obstacle with
conductivity boundary conditions by Hettlich \cite{Hettlich96}.
However, few results are available for the case of a piecewise homogeneous
background medium. For the scalar Helmholtz equation, uniqueness
results have been investigated in our recent paper \cite{LZsiam}.
Motivated by Isakov's paper \cite{Isakov90}, Kirsch \& Kress
\cite{KK93} gave another proof of the unique determination of a penetrable
obstacle in a homogeneous medium for the scalar transmission
problem. Based on their ideas, H\"{a}hner proved the unique
determination of the penetrable obstacle in the electromagnetic scattering
using a novel method. In this paper, we will use H\"{a}hner's method to prove
that the interface $S_{0}$ is uniquely determined from the electric far field
patterns in Section \ref{sec3}. Under the condition that the wave
numbers in the innermost and outermost homogeneous layers coincide,
that is, $k_{0}=k_{1}$ in the Maxwell equations (\ref{0EH})-(\ref{0FG})
and $S_{0}$ is known in advance, we have proved in \cite{LZaa} that the obstacle
with its physical property can be uniquely determined.
A main tool used and established in \cite{LZaa} is a mixed reciprocity relation
which is important in the proof of the uniqueness result.
In this paper, we establish a modified mixed reciprocity relation
(cf. Lemma 2.5 in \cite{LZip} for the acoustic scattering).
Based on this result, we will prove the unique determination of the obstacle $\Om_{2}$
and its physical property $\mathscr{B}$ in Section \ref{sec4}.

\section{The direct scattering problem}\label{sec2}
\setcounter{equation}{0}

We first show that the direct scattering problem has a unique
solution.

\begin{theorem}\label{uni.direct}
The boundary value problem $(\ref{0EH})-(\ref{0ibc})$ admits at most
one solution.
\end{theorem}
\begin{proof}
Clearly, it is enough to show that $E=H=0$ in $\Om_{0}$, $F=G=0$ in
$\Om_1$ for the corresponding homogeneous problem, that is,
$T_{1}=T_{2}=0$ on $S_{0}$, $T_{3}=0$ on $\G_{1}$ and $T_{4}=0$ on
$\G_{2}$. Using Green's vector theorem, we have
 \be\label{uni1}
\int_{S_{0}}\nu\times E\cdot \ov{H}ds &=&\int_{S_{0}}\nu\times
E\cdot [(\nu\times\ov{H})\times\nu]ds\cr
&=&\la_{E}\la_{H}\int_{S_{0}}\nu\times F\cdot
[(\nu\times\ov{G})\times\nu]ds\cr
&=&\la_{E}\la_{H}\int_{S_{0}}\nu\times F\cdot\ov{G}ds\cr
&=&\la_{E}\la_{H}\int_{S_{1}}\nu\times F\cdot\ov{G}ds
  +\la_{E}\la_{H}\int_{\Om_{1}}(\curl F\cdot\ov{G}-F\cdot\curl\ov{G})dx\cr
&=&-\frac{\la_{E}\la_{H}}{\ov{k_{1}}}\int_{\G_{2}}\la|\nu\times
 F|^{2}ds+i\la_{E}\la_{H}\int_{\Om_{1}}(k_{1}|G|^{2}-\ov{k_{1}}|F|^{2})dx
 \en
where we have used the transmission boundary conditions (\ref{0tr})
in the second equality, the Maxwell equations
(\ref{0EH})-(\ref{0FG}) and the boundary conditions
(\ref{0pbc})-(\ref{0ibc}) in the fifth equality. Taking the real
part of (\ref{uni1}), we have, on noting that
$\la_{E}\la_{H}=k_{0}/k_{1}$, $\Re k_{1}>0$, $\Im k_{1}\geq0$ and
$\la>0$ that
 \ben
\Re\int_{S_{0}}\nu\times E\cdot\ov{H}ds
=-\frac{k_{0}}{|k_{1}|^{2}}\int_{\G_{2}}\la|\nu\times F|^{2}ds
 -\frac{2k_{0}\Re k_{1}\Im k_{1}}{|k_{1}|^{2}}\int_{\Om_{1}}|F|^{2}ds\leq0.
 \enn
Therefore, by Rellich's lemma \cite{CK}, it follows that $E=H=0$ in
$\Om_{0}$. The transmission boundary conditions (\ref{0tr}) and
Holmgren's uniqueness theorem \cite{Kress001} imply that $F=G=0$ in
$\Om_{1}$, which completes the proof.
\end{proof}

Denote by $\Phi_{j}$ the fundamental solution of the Helmholtz
equation with wave number $k_{j}\;(j=0,1)$, which is given by
 \be\label{Phidef}
 \Phi_{j}(x,y)=\frac{e^{ik_{j}|x-y|}}{4\pi |x-y|}, \qquad \ x,y\in \R^{3},x\neq y.
 \en
For convenience, let $\G_{0}$ denote the interface $S_{0}$. Given two
integrable vector fields $a$ on $S_{0}$ and $b$ on $\G_{j}$ and an
integral function $\psi$ on $\G_{j},\;(j=1,2)$, we introduce the
integral operators $M_{i,j}$ and $N_{i,j}$ , respectively, by
 \ben
 (M_{i,j}a)(x)=2\int_{S_{0}}\nu(x)\times\curl_{x}\{a(y)\Phi_{j}(x,y)\}ds(y), \qquad x\in\;\G_{i}\\
 (N_{i,j}a)(x)=2\nu(x)\times\curl\curl\int_{S_{0}}\nu(y)\times a(y)\Phi_{j}(x,y)ds(y),
                                                              \qquad x\in\;\G_{i}
 \enn
for $i=0,1,2; j=0,1$, the operators $\wi{M}_{i,j}$ and
$\wi{N}_{i,j}$, respectively, by
 \ben
 (\wi{M}_{i,j}b)(x)=2\int_{\G_{j}}\nu(x)\times\curl_{x}\{b(y)\Phi_{1}(x,y)\}ds(y), \qquad x\in\;\G_{i}\\
 (\wi{N}_{i,j}b)(x)=2\nu(x)\times\curl\curl\int_{\G_{j}}\nu(y)\times b(y)\Phi_{1}(x,y)ds(y),
                                                              \qquad x\in\;\G_{i}
 \enn
for $i=0,1,2, j=1,2$ and the operators $\wi{S}_{i,j}$ and
$\wi{K}_{i,j}$, respectively, by
 \ben
 (\wi{S}_{i,j}\psi)(x)=2\int_{\G_{j}}\Phi_{1}(x,y)\psi(y)ds(y), \qquad x\in\;\G_{i}\\
 (\wi{K}_{i,j}\psi)(x)=2\int_{\G_{j}}\frac{\pa\Phi_{1}(x,y)}{\pa\nu(y)}\psi(y)ds(y),
                                                              \qquad x\in\;\G_{i}
 \enn
for $i=0,1,2, j=1,2$.

\begin{theorem}\label{well-posedness}
The boundary value problem $(\ref{0EH})-(\ref{0ibc})$ has a unique
solution. The solution depends continuously on the boundary data in
the sense that the operator mapping the given boundary data onto the
solution is continuous from $T_{d}^{0,\al}(S_{0})\times
T_{d}^{0,\al}(S_{0})\times T_{d}^{0,\al}(\G_{1})\times
T^{0,\al}(\G_{2})$ into $C^{0,\al}(\ov{\Om_{0}})\times
C^{0,\al}(\ov{\Om_{0}})\times C^{0,\al}(\ov{\Om_{1}})\times
C^{0,\al}(\ov{\Om_{1}})$.
\end{theorem}

\begin{proof}
The uniqueness of solutions follows from Theorem \ref{uni.direct}.
We now prove the existence of solutions using the integral equation method.
We seek a solution in the form
 \be
\label{E}E(x)&=&\frac{\la_{H}k_{0}}{k_{1}}\curl\int_{S_{0}}a(y)\Phi_{0}(x,y)ds(y)
        +\la_{E}\curlcurl\int_{S_{0}}b(y)\Phi_{0}(x,y)ds(y),\qquad\qquad\\
\label{H}H(x)&=&\frac{1}{ik_{0}}\curl E(x)\cr
     &=&\frac{\la_{H}}{ik_{1}}\curlcurl\int_{S_{0}}a(y)\Phi_{0}(x,y)ds(y)
       +\frac{\la_{E}k_{0}}{i}\curl\int_{S_{0}}b(y)\Phi_{0}(x,y)ds(y)\qquad\qquad
 \en
for $x\in \Om_{0}$ and
 \be
\label{F}F(x)&=&\curl\int_{S_{0}}a(y)\Phi_{1}(x,y)ds(y)
                +\curlcurl\int_{S_{0}}b(y)\Phi_{1}(x,y)ds(y)\cr
   &&+\curl\int_{\G_{1}}c(y)\Phi_{1}(x,y)ds(y)
    +\frac{i}{k_{1}^{2}}\curlcurl\int_{\G_{1}}\nu(y)
    \times(\widehat{S}_{1}^{2}c)(y)\Phi_{1}(x,y)ds(y)\cr
  &&+\int_{\G_{2}}d(y)\Phi_{1}(x,y)ds(y)+i\la\curl\int_{\G_{2}}\nu(y)
    \times(\widehat{S}_{2}^{2}d)(y)\Phi_{1}(x,y)ds(y)\cr
   &&+{\rm grad}\int_{\G_{2}}\psi(y)\Phi_{1}(x,y)ds(y)
     +i\la\int_{\G_{2}}\nu(y)\psi(y)\Phi_{1}(x,y)ds(y), \\
\label{G}G(x)&=&\frac{1}{ik_{1}}\curl F(x)\cr
   &=&\frac{1}{ik_{1}}\curlcurl\int_{S_{0}}a(y)\Phi_{1}(x,y)ds(y)
      -ik_{1}\curl\int_{S_{0}}b(y)\Phi_{1}(x,y)ds(y)\cr
   &&+\frac{1}{ik_{1}}\curlcurl\int_{\G_{1}}c(y)\Phi_{1}(x,y)ds(y)\cr
   &&+\frac{1}{k_{1}}\curl\int_{\G_{1}}\nu(y)
     \times(\widehat{S}_{1}^{2}c)(y)\Phi_{1}(x,y)ds(y)\cr
   &&+\frac{1}{ik_{1}}\curl\int_{\G_{2}}d(y)\Phi_{1}(x,y)ds(y)\cr
   &&+\frac{\la}{k_{1}}\curlcurl\int_{\G_{2}}\nu(y)
     \times(\widehat{S}_{2}^{2}d)(y)\Phi_{1}(x,y)ds(y)\cr
   &&+\frac{\la}{k_{1}}\curl\int_{\G_{2}}\nu(y)\psi(y)\Phi_{1}(x,y)ds(y)
 \en
for $x\in \Om\setminus S_{1},$ where $a, b\in T_{d}^{0,\al}(S_{0}),
c\in T_{d}^{0,\al}(\G_{1}), d\in T^{0,\al}(\G_{2})$ and $\psi\in
C^{0,\al}(\G_{2})$ are five densities to be determined and
$\wi{S}_{j}$ is the single-layer operator given by
 \ben
 (\widehat{S}_{j}c)(x):=\frac{1}{2\pi}\int_{\G_{j}}\frac{1}{|x-z|}c(z)ds(z),
     \qquad x\in \G_{j},\;j=1,2.
 \enn

If the densities $a, b\in T_d(S_0),c\in T_d^{0,\al}(\G_1),d\in T^{0,\al}(\G_{2})$
and $\psi\in C^{0,\al}(\G_{2})$ satisfy the following system of integral equations
(\ref{T1})-(\ref{divF=0}), then we can conclude from the integral equations (\ref{T1}) and
(\ref{T2}) that $a, b\in T_{d}^{0,\al}(S_{0})$. Therefore we choose
the solution space $X:=T_{d}(S_{0})\times T_{d}(S_{0})\times
T_{d}^{0,\al}(\G_{1})\times T^{0,\al}(\G_{2})\times C^{0,\al}(\G_{2})$
for the system of integral equations (\ref{T1})-(\ref{divF=0}) below.
The vector field $F$ given by (\ref{F}) clearly satisfies the vector Helmholtz
equation and its cartesian components satisfy the Sommerfeld radiation condition.
Hence, if we insist that ${\rm div} F=0$ in $\R^{3}\setminus \G_2$, then by
Theorems 6.4 and 6.7 in \cite{CK} we have that $F,G$ satisfy the
Maxwell equations (\ref{0FG}). Since ${\rm div} F$ satisfies the
scalar Helmholtz equation and the Sommerfeld radiation condition, by
the uniqueness for the exterior Dirichlet problem (see Theorem 3.7
in \cite{CK}) it suffices to impose ${\rm div} F=0$ only on the
boundary $\G_2$. Then the jump relations together with the
regularity results of surface potentials imply that $E,H,F,G$
defined in (\ref{E})-(\ref{G}) solve the direct problem provided the
densities $a,b,c,d,\psi$ solve
 \be
\label{T1}\left(\la_{E}+\frac{\la_{H}k_{0}}{k_{1}}\right)a
           +L_{1}a+M_{1}b+N_{1}c+P_{1}d+Q_{1}\psi&=&2T_{1}\quad {\rm on}\;S_{0}, \\
\label{T2}\frac{\la_{E}k_{0}+\la_{H}k_{1}}{i}b+L_{2}a+M_{2}b
           +N_{2}c+P_{2}d+Q_{2}\psi&=&2T_{2}\quad {\rm on}\;S_{0}, \\
\label{T3}c+L_{3}a+M_{3}b+N_{3}c+P_{3}d+Q_{3}\psi&=&2T_{3}\quad {\rm on}\;\;\G_{1},\\
\label{T4}d+L_{4}a+M_{4}b+N_{4}c+P_{4}d+Q_{4}\psi&=&2T_{4}\quad {\rm on}\;\;\G_{2},\\
\label{divF=0}i\la\psi+P_{5}d+Q_{5}\psi&=&0\qquad {\rm on}\;\;\G_{2}
 \en
where
 \ben
L_{1}&:=&\frac{\la_{H}k_{0}}{k_{1}}M_{0,0}-\la_{E}M_{0,1},\hspace{2.5cm}
M_{1}:=\la_{E}(N_{0,0}-N_{0,1})R,\\
N_{1}&:=&-\la_{E}(\wi{M}_{0,1}+\frac{i}{k_{1}^{2}}\wi{N}_{0,1}P\widehat{S}_{1}^{2}),\hspace{1.9cm}
P_{1}:=\la_{E}R\wi{S}_{0,2}+i\la\la_{E}\wi{M}_{0,2}R\widehat{S}_{2}^{2},\\
Q_{1}\psi&:=&-2\la_{E}\nu(x)\times{\rm grad}
             \int_{\G_{2}}\Phi_{1}(x,y)\psi(y)ds(y)+i\la\la_{E}R\wi{S}_{0,2}(\nu\psi),\\
L_{2}&:=&\frac{\la_{H}}{ik_{1}}(N_{0,0}-N_{0,1})R,\hspace{2.9cm}
M_{2}:=\frac{\la_{E}k_{0}}{i}M_{0,0}-\frac{\la_{H}k_{1}}{i}M_{0,1},\;\;\\
N_{2}&:=&-\frac{\la_{H}}{ik_{1}}\wi{N}_{0,1}R
         +\frac{\la_{H}}{k_{1}}\wi{M}_{0,1}R\widehat{S}_{1}^{2},\hspace{1.1cm}
P_{2}:=\frac{i\la_{H}}{k_{1}}\wi{M}_{0,2}
     -\frac{\la\la_{H}}{k_{1}}\wi{N}_{0,2}P\widehat{S}_{2}^{2},\\
Q_{2}\psi&:=&-2\frac{\la\la_{H}}{k_{1}}\nu(x)\times\curl\int_{\G_{2}}\nu(y)\psi(y)\Phi_{1}(x,y)ds(y),\\
L_{3}&:=&M_{1,1},\qquad\qquad M_{3}:=N_{1,1}R,\hspace{1.5cm}
N_{3}:=\wi{M}_{1,1}+\frac{i}{k_{1}^{2}} \wi{N}_{1,1}P\widehat{S}_{1}^{2},\\
P_{3}&:=&-R\wi{S}_{1,2}-i\la\wi{M}_{1,2}R\widehat{S}_{2}^{2},\\
Q_{3}\psi:&=&2\nu(x)\times{\rm
    grad}\int_{\G_{2}}\Phi_{1}(x,y)\psi(y)ds(y)-i\la R\wi{S}_{1,2}(\nu\psi),\\
L_{4}&:=&N_{2,1}-i\la RM_{2,1},\hspace{3.5cm} M_{4}:=k_{1}^{2}M_{2,1}-i\la RN_{2,1}R,\\
N_{4}&:=&\wi{N}_{2,1}R-i\wi{M}_{2,1}R\widehat{S}_{1}^{2}
     -i\la R\wi{M}_{2,1}+\frac{\la}{k_{1}^{2}} R\wi{N}_{2,1}P\widehat{S}_{1}^{2},\\
P_{4}&:=&\wi{M}_{2,2}+i\la\wi{N}_{2,2}P\widehat{S}_{2}^{2}-i\la
P\wi{S}_{2,2}
     -\la^{2}R\wi{M}_{2,2}R\widehat{S}_{2}^{2}+\la^{2}P\widehat{S}_{2}^{2},\\
Q_{4}\psi&:=&2i\la\nu(x)\times\int_{\G_{2}}{\rm
grad}_{x}\Phi_{1}(x,y)\times\{\nu(y)-\nu(x)\}\psi(y)ds(y)
     +\la^{2}P\wi{S}_{2,2}(\nu\psi),\\
P_{5}d&:=&-2\int_{\G_{2}}{\rm grad}_{x}\Phi_{1}(x,y)\cdot
d(y)ds(y),\hspace{1cm} Q_{5}=k_{1}^{2}\wi{S}_{2,2}+i\la\wi{K}_{2,2}.
 \enn
Here $R, P$ are defined by $Ra:=a\times\nu$ and $Pa:=(\nu\times a)\times\nu,$ respectively.

Writing the system of integral equations (\ref{T1})-(\ref{divF=0})
in the matrix form
 \ben
 \left(\begin{matrix}\la_{a}&0&0&0&0\\
      0&\la_{b}&0&0&0\\
      0&0&1&0&0\\0&0&0&1&0\\0&0&0&P_{5}&i\la
 \end{matrix}\right)
 \left(\begin{matrix}a\\b\\c\\d\\\psi
 \end{matrix}\right)
+\left(\begin{matrix}L_{1}&M_{1}&N_{1}&P_{1}&Q_{1}\\ L_{2}&M_{2}&N_{2}&P_{2}&Q_{2}\\
      L_{3}&M_{3}&N_{3}&P_{3}&Q_{3}\\L_{4}&M_{4}&N_{4}&P_{4}&Q_{4}\\0&0&0&0&Q_{5}
 \end{matrix}\right)
 \left(\begin{matrix}a\\b\\c\\d\\\psi
 \end{matrix}\right)
=\left(\begin{matrix}2T_{1}\\2T_{2}\\2T_{3}\\2T_{4}\\0
 \end{matrix}\right),
 \enn
where $\la_{a}=\la_{E}+(\la_{H}k_{0})/k_{1}$ and $\la_{b}=-i\la_{E}k_{0}-i\la_{H}k_{1}$,
we can see that the first matrix operator has a bounded inverse because of its
triangular form and the second one is compact by Theorems 6.15 and 6.16 in \cite{CK}
and the argument in the proof of Theorems 6.19 and 9.12 in \cite{CK}. Therefore
the Riesz-Fredholm theory can be applied to prove the existence of a unique solution to the
system (\ref{T1})-(\ref{divF=0}) in the solution space $X$.
To do this, let $a, b, c, d$ and $\psi$ be a solution to the homogeneous form
of (\ref{T1})-(\ref{divF=0}) (i.e., $T_{1}=T_{2}=0$ on $S_{0}$, $T_{3}=0$ on $\G_{1}$ and
$T_{4}=0$ on $\G_{2}$). We then need to show that $a=b=c=d=\psi=0$.
First, by Theorem \ref{uni.direct} we conclude that $E=H=0$ in $\Om_{0}$ and $F=G=0$ in $\Om_{1}$.

Now, by the jump relations we have
 \be
\label{gama1}-\nu\times F_{-}=c,\quad-\nu\times G_{-}
              =\frac{1}{k_{1}}\nu\times\widehat{S}_{1}^{2}c\qquad{\rm on}\;\;\G_{1},\\
\label{gama21}-\nu\times F_{-}=i\la\nu\times
 \widehat{S}_{2}^{2}d,\quad -\nu\times G_{-}=\frac{1}{ik_{1}}d\qquad{\rm on}\;\; \G_{2},\\
\label{gama22} -{\rm div}F_{-}=-i\la\psi,\quad -\nu\cdot F_{-}=-\psi\qquad {\rm on}\;\; \G_{2}\;
 \en
where $F_{-},\; G_{-}$ denote the limit of $F,\;G$ on the surface
$S_{1}$ from the interior of $S_{1}$, respectively. Hence, using Green's
vector theorem, we derive from (\ref{gama1})-(\ref{gama22}) that
 \ben
 &&i\int_{\G_{1}}|\widehat{S}_{1}c|^{2}ds+i\la\int_{\G_{2}}|\widehat{S}_{2}d|^{2}ds
   +i\la\int_{\G_{2}}|\psi|^{2}ds\cr
 &=&\int_{S_{1}}\{\nu\times \ov{F_{-}}\cdot \curl F+\nu\cdot \ov{F_{-}}{\rm div}F\}ds\cr
 &=&\int_{\Om_{2}}\{|\curl F|^{2}+|{\rm div}F|^{2}-[(\Re k_{1})^{2}
       -(\Im k_{1})^{2}]|F|^{2}-2i\Re k_{1}\Im k_{1}|F|^{2}\}dx.
 \enn
Taking the imaginary part of the last equation gives that $\psi=0$ on $\G_{2}$,
$\widehat{S}_{1}c=0$ on $\G_{1}$ and $\widehat{S}_{2}d=0$ on $\G_{2}$.
This implies that $c=0$ on $\G_{1}$, $d=0$ on $\G_{2}$ (see the proof
of Theorem 3.10 in \cite{CK}).

On the other hand, define
 \be
\label{Enew}\wi{E}(x)&=&-m\curl\int_{S_{0}}a(y)\Phi_{0}(x,y)ds(y)
        -\frac{m\la_{E}k_{1}}{\la_{H}k_{0}}\curlcurl\int_{S_{0}}b(y)\Phi_{0}(x,y)ds(y),\;\;\qquad\\
\label{Hnew}\wi{H}(x)&=&\frac{1}{ik_{0}}\curl \wi{E}(x)\cr
     &=&\frac{i m}{k_{0}}\curlcurl\int_{S_{0}}a(y)\Phi_{0}(x,y)ds(y)
       -\frac{m\la_{E}k_{1}}{i\la_{H}}\curl\int_{S_{0}}b(y)\Phi_{0}(x,y)ds(y)\qquad
 \en
for $x\in\Om$ and
 \be
\label{Fnew}\wi{F}(x)&=&\curl\int_{S_{0}}a(y)\Phi_{0}(x,y)ds(y)
            +\curlcurl\int_{S_{0}}b(y)\Phi_{0}(x,y)ds(y),\;\;\qquad\\
\label{Gnew}\wi{G}(x)&=&\frac{1}{ik_{1}}\curl \wi{F}(x)\cr
   &=&\frac{1}{ik_{1}}\curlcurl\int_{S_{0}}a(y)\Phi_{0}(x,y)ds(y)
            -ik_{1}\curl\int_{S_{0}}b(y)\Phi_{0}(x,y)ds(y)\qquad
 \en
for $x\in \Om_{0}$, where $m$ is a constant given by
$m:=\sqrt{\la_{H}/\la_{E}}$. Then by the jump relations we have
 \be
\label{11tr}\nu\times\wi{F}-\nu\times F=a,\qquad
  \frac{k_{1}}{\la_{H}k_{0}}\nu\times E+\frac{1}{m}\nu\times\wi{E}
     =a\qquad&& {\rm on}\;S_{0},\qquad \\
\label{12tr}\nu\times\wi{G}-\nu\times G=-ik_{1}b,\qquad
  \frac{k_{1}}{\la_{E}k_{0}}\nu\times H
     +\frac{\la_{H}}{m\la_{E}}\nu\times\wi{H}=-ik_{1}b\qquad&& {\rm on}\;S_{0}.
 \en
Thus, $\wi{E},\wi{H},\wi{F},\wi{G}$ given by (\ref{Enew})-(\ref{Gnew}) solves
the homogeneous transmission problem
 \be
 \label{1FG}\curl\wi{F}-ik_1\wi{G}=0,\qquad\curl\wi{G}+ik_1\wi{F}=0\quad&&\mbox{in}\;\;\Om_0,\\
 \label{1EH}\curl\wi{E}-ik_0\wi{H}=0,\qquad\curl\wi{H}+ik_0\wi{E}=0\quad&&\mbox{in}\;\;\Om,\\
 \label{1tr}\nu\times \wi{F}-\frac{1}{m}\nu\times \wi{E}=0,\qquad
   \nu\times \wi{G}-\frac{\la_{H}}{m\la_{E}}\nu\times \wi{H}=0 \quad&&\mbox{on}\;\;S_0,\\
 \label{1SMrc} \lim_{|x|\rightarrow\infty}(\wi{H}\times x-|x|\wi{E})=0\quad&&
 \en
where the limit holds uniformly in all directions $x/|x|$.
Similar to the argument as in the proof of Theorem \ref{uni.direct},
we find on using Green's vector theorem and by the definition of $m$ that
 \ben\label{uni2}
\int_{S_{0}}\nu\times \wi{F}\cdot \ov{\wi{G}}ds&=&\int_{S_{0}}\nu\times \wi{F}\cdot
  [(\nu\times\ov{\wi{G}})\times\nu]ds=\int_{S_{0}}\nu\times \wi{E}
  \cdot [(\nu\times\ov{\wi{H}})\times\nu]ds\cr
&=&\int_{\Om}(\curl\wi{E}\cdot\ov{\wi{H}}-\wi{E}\cdot\curl\ov{\wi{H}})dx
  =ik_{0}\int_{\Om_{1}}(|\wi{H}|^{2}-|\wi{E}|^{2})dx
 \enn
Taking the real part of the above equation, we have
 \ben
\Re\int_{S_{0}}\nu\times \wi{F}\cdot \ov{\wi{G}}ds\leq0.
 \enn
Therefore, by Rellich's lemma (see Theorem 4.17 in \cite{CK83}), it
follows that $\wi{F}=\wi{G}=0$ in $\Om_{0}$. The transmission
boundary conditions (\ref{1tr}) and Holmgren's uniqueness theorem \cite{Kress001}
imply that $\wi{E}=\wi{H}=0$ in $\Om$.
Hence, by the relations (\ref{11tr}) and (\ref{12tr}), we conclude that $a=b=0$ on
$S_{0}$, which completes the proof.
\end{proof}

As incident fields $E^{i}$, besides the electromagnetic plane waves
we are also interested in the electromagnetic field of an electric
dipole with polarization $p$ which is given by
 \be\label{Eied}
E^{i}(x;z_{j},p)=\frac{i}{k_{j}}\curlcurl(p\Phi(x,z_{j})),\;\;\;
H^{i}(x;z_{j},p)=\curl(p\Phi_{j}(x,z_{j}))
 \en
located at the point $z_{j}\in\Om_{j}$ for $x\neq z_{j}$ $(j=0,1)$.
For $j=0,1$, denote by $E(\cdot;z_{j},p)$ and $H(\cdot;z_{j},p)$ the
corresponding total waves, by $E^{s}(\cdot;z_{j},p)$ and
$H^{s}(\cdot;z_{j},p)$ the scattered waves and by
$E^{\infty}(\cdot;z_{j},p)$ and $H^{\infty}(\cdot;z_{j},p)$ the
corresponding far field patterns, indicating the dependence on the
location $z_{j}$ and polarization $p$ of the electric dipole.

\begin{remark}\label{ele} {\rm
In the case when the incident field is given by the electric dipole
$E^{i}(x;z_{0},p), H^{i}(x;z_{0},p)$ located at $z_{0}\in\Om_{0}$,
we have $E=E^{s}(x;z_{0},p),\; T_{1}=-\nu\times E^{i}(x;z_{0},p),\;
T_{2}=-\nu\times H^{i}(x;z_{0},p),\;T_{3}=0$ and $T_{4}=0$ in the
boundary value problem $(\ref{0EH})-(\ref{0ibc})$. In the case when the incident
field is given by the electric dipole $E^{i}(x;z_{1},p), H^{i}(x;z_{1},p)$
located at $z_{1}\in\Om_{1}$, we have
$E=E(x;z_{1},p),\; T_{1}=\la_{E}\nu\times E^{i}(x;z_{1},p),\;
T_{2}=\la_{H}\nu\times H^{i}(x;z_{1},p),\;T_{3}=-\nu\times E^{i}(x;z_{1},p)$ and
$T_{4}=-\nu\times H^{i}(x;z_{1},p)+\frac{\la}{k_{1}}(\nu\times
E^{i}(x;z_{1},p))\times\nu$ in the boundary value problem
$(\ref{0EH})-(\ref{0ibc})$. Note that in the last case
$E^{\infty}(\cdot;z_{1},p)$ is the electric far field pattern of the
field $E(x;z_{1},p)=E^{s}(x;z_{1},p)$ for $x\in \Om_{0}$. This is
different from our early work \cite{LZaa}, where the total field
$E(x;z_{1},p)$ in $\Om_{0}$ is decomposed into the scattered field
$E^{s}(x;z_{1},p)$ and the incident field $E^{i}(x;z_{1},p)$, so
$E^{\infty}(\cdot;z_{1},p)$ is the electric far field pattern of the
scattered field $E^{s}(x;z_{1},p)$ rather than the total field $E(x;z_{1},p)$.
}
\end{remark}

We now establish the following mixed reciprocity relation which is
needed for the inverse problem.

\begin{lemma}\label{mix}
{\rm(Mixed reciprocity relation)} For scattering of plane waves
$E^{i}(x,d,q)$ and electric dipole $E^{i}(x;z,p)$ from the obstacle
$\Om_{2}$ we have
 \be
\label{mixz0}4\pi q\cdot E^{\infty}(-d;z_0,p)=p\cdot E^s(z_0,d,q),\qquad z_{0}\in\Om_0,\\
\label{mixz1}4\pi q\cdot E^{\infty}(-d;z_1,p)
      =\la_{E}\la_{H}p\cdot F(z_{1},d,q),\qquad z_{1}\in\Om_{1}
 \en
for all incident directions $d\in S^{2}$ and all polarizations $p,q\in\R^{3}$.
\end{lemma}

\begin{proof}
Arguing similarly as in the proof of Lemma 3.1 in \cite{LZaa}, we
have
 \be\label{qEinf2}
&&4\pi q\cdot E^{\infty}(-d;z,p)\cr
&&=\int_{S_{0}}\{\nu(y)\times E^{s}(y;z,p)\cdot H(y,d,q)
   +\nu(y)\times H^{s}(y;z,p)\cdot E(y,d,q)\}ds(y)
 \en
for all $z\in\Om_{0}\cup\Om_{1}$ and all $d\in S^{2},\;p,q\in\R^{3}$.

We first consider the case $z:=z_{0}\in\Om_{0}$. Using a similar argument
as in the proof of Lemma 3.1 in \cite{LZaa}, we have
 \be\label{ps}
&&p\cdot E^{s}(z,d,q)\cr
&&=\int_{S_{0}}\{\nu(y)\times H(y,d,q)\cdot E^{i}(y;z,p)
  -\nu(y)\times H^{i}(y;z,p)\cdot E(y,d,q)\}ds(y)
 \en
for all $z\in\Om_{0}$ and all $d\in S^{2},\;p,q\in \R^{3}$.
Subtract the equation (\ref{ps}) from the equation (\ref{qEinf2}) to obtain that
 \ben
 &&4\pi q\cdot E^{\infty}(-d;z,p)-p\cdot E^{s}(z,d,q)\cr
&&=\int_{S_{0}}\{\nu(y)\times E(y;z,p)\cdot H(y,d,q)
   +\nu(y)\times H(y;z,p)\cdot E(y,d,q)\}ds(y)\cr
&&=\la_{E}\la_{H}\int_{S_{0}}\{\nu(y)\times F(y;z,p)\cdot G(y,d,q)
   +\nu(y)\times G(y;z,p)\cdot F(y,d,q)\}ds(y)\cr
&&=\la_{E}\la_{H}\int_{S_{1}}\{\nu(y)\times F(y;z,p)\cdot G(y,d,q)
   +\nu(y)\times G(y;z,p)\cdot F(y,d,q)\}ds(y)\cr
&&=0
 \enn
for all $z\in\Om_{0}$ and all $d\in S^{2},\;p,q\in\R^{3}$,
where we have used the transmission boundary condition (\ref{0tr})
in the second equality, Green's vector theorem and Maxwell's equations
(\ref{0FG}) in the third equality and the boundary conditions
(\ref{0pbc}) and (\ref{0ibc}) in the fourth equality.

We now consider the case $z:=z_{1}\in\Om_1$. Using the transmission
boundary condition (\ref{0tr}) and  Green's vector theorem, we deduce that
 \be\label{qEinf3}
&&4\pi q\cdot E^{\infty}(-d;z,p)\cr
&=&\la_{E}\la_{H}\int_{S_{0}}\{\nu(y)\times E^{i}(y;z,p)\cdot
G(y,d,q)+\nu(y)\times H^{i}(y;z,p)\cdot F(y,d,q)\}ds(y)\cr &&
+\la_{E}\la_{H}\int_{S_{0}}\{\nu(y)\times F(y;z,p)\cdot
G(y,d,q)+\nu(y)\times G(y;z,p)\cdot F(y,d,q)\}ds(y)\cr
&=&\la_{E}\la_{H}\int_{S_{0}}\{\nu(y)\times E^{i}(y;z,p)\cdot
G(y,d,q)+\nu(y)\times H^{i}(y;z,p)\cdot F(y,d,q)\}ds(y)\cr &&
+\la_{E}\la_{H}\int_{S_{1}}\{\nu(y)\times F(y;z,p)\cdot
G(y,d,q)+\nu(y)\times G(y;z,p)\cdot F(y,d,q)\}ds(y)\qquad\quad
 \en
for all $z\in\Om_{1}$ and all $d\in S^{2},\;p,q\in \R^{3}$.
From the Stratton-Chu formula (see Theorem 6.2 in \cite{CK}):
 \ben\label{es2}
F(z,d,q)&=&-\curl\int_{S_{0}}\nu(y)\times F(y,d,q)\Phi_{1}(z,y)\cr
  &&+\frac{1}{ik_{1}}\curlcurl\int_{S_{0}}\nu(y)\times G(y,d,q)\Phi_{1}(z,y)ds(y)\cr
&&+\curl\int_{S_{1}}\nu(y)\times F(y,d,q)\Phi_{1}(z,y)\cr
 &&-\frac{1}{ik_{1}}\curlcurl\int_{S_{1}}\nu(y)\times G(y,d,q)\Phi_{1}(z,y)ds(y),
 \enn
and with the help of the vector identities
 \be
\label{vi}p\cdot\curl_{z}\curl_{z}(a(y)\Phi(z,y))&=&a(y)\cdot\curl_{y}\curl_{y}(p\Phi(z,y)),\\
\label{vi2}p\cdot\curl_{z}(a(y)\Phi(z,y))&=&a(y)\cdot\curl_{y}(p\Phi(z,y)),
 \en
it follows that
 \ben\label{ps2}
&&\la_{E}\la_{H}p\cdot F(z,d,q)\cr
&=&-\la_{E}\la_{H}\int_{S_{0}}\{\nu(y)\times F(y,d,q)\cdot H^{i}(y;z,p)
   +\nu(y)\times G(y,d,q)\cdot E^{i}(y;z,p)\}ds(y)\cr
&&+\la_{E}\la_{H}\int_{S_{1}}\{\nu(y)\times F(y,d,q)\cdot H^{i}(y;z,p)
 +\nu(y)\times G(y,d,q)\cdot E^{i}(y;z,p)\}ds(y)
 \enn
for all $z\in\Om_{1}$ and all $d\in S^{2},\;p,q\in \R^{3}$. We now
subtract the last equation from the equation (\ref{qEinf3}) to obtain
 \ben
 &&4\pi q\cdot E^{\infty}(-d;z_{1},p)-\la_{E}\la_{H}p\cdot F(z_{1},d,q)\cr
 &=&\la_{E}\la_{H}\int_{S_{1}}\{\nu(y)\times [F(y;z,p)+E^{i}(y;z,p)]\cdot G(y,d,q)\cr
  && +\nu(y)\times [G(y;z,p)+H^{i}(y;z,p)]\cdot F(y,d,q)\}ds(y).
 \enn
This, together with the boundary conditions
 \ben
 \nu(y)\times [F(y;z,p)+E^{i}(y;z,p)]=\nu(y)\times F(y,d,q)=0\qquad {\rm on}\;\; \G_{1}
 \enn
and
 \ben
 &&\nu(y)\times
 [G(y;z,p)+H^{i}(y;z,p)]-\frac{\la}{k_{1}}
    [\nu(y)\times(F(y;z,p)+E^{i}(y;z,p))]\times\nu(y)\cr
 &=&\nu(y)\times G(y,d,q)-\frac{\la}{k_{1}}
    [\nu(y)\times F(y,d,q)]\times\nu(y)=0 \qquad {\rm on}\;\; \G_{2},
 \enn
gives the desired result (\ref{mixz1}).
\end{proof}


\section{Unique determination of the interface $S_{0}$}\label{sec3}
\setcounter{equation}{0}

In this section we will assume that $\Im k_{1}>0$. Denote by $S_{0}$
and $\wi{S}_{0}$ two different interfaces which yield the same electric
far field patterns for all plane incoming waves, that is,
$E^{\infty}(\widehat{x},d,q)=\wi{E}^{\infty}(\widehat{x},d,q)$ for
all $\widehat{x},d\in S^{2},\;q\in \R^{3}$. Let $\wi{\Om}$ and
$\wi{\Om_{0}}$ be the bounded and unbounded domains with interface
$\wi{S}_{0}$, respectively. We denote by $B$ a large ball containing
$\ov{\Om}\cup\ov{\wi{\Om}}$ and by $\Om_{e}$ the connected component
of $\Om_{0}\cap\wi{\Om_{0}}$ that contains the exterior of $B$. If
$S_{0}\neq\wi{S}_{0}$, we may assume without loss of generality that there is a
point $x^{\ast}\in S_{0}$ such that $x^{\ast}\in \wi{\Om_{0}}$ and
$x^{\ast}\in \pa\Om_{e}$ (the case with $x^{\ast}\in \wi{S}_{0}$ and
$x^{\ast}\in \Om_{0}$ can be treated similarly). Assume that
$B_{\delta_{1}}\subset B$ is a ball centered at $x^{\ast}$ with
sufficiently small radius $\delta_{1}>0$ such that
$B_{\delta_{1}}\cap \ov{\wi{\Om}}=\emptyset$, as shown in Figure \ref{fig2}.
For fixed $\delta_{1}$ we have $B_{\delta_{1}}\cap\Om_{0}
+\delta_{2}\nu(x^{\ast})\subset \Om_{e}$ and
$B_{\delta_{1}}\cap S_{0}-\delta_{2}\nu(x^{\ast})\subset \Om$ for
all sufficiently small $\delta_{2}>0$.

\begin{figure}[htbp]
\centering \scalebox{0.38}{\includegraphics{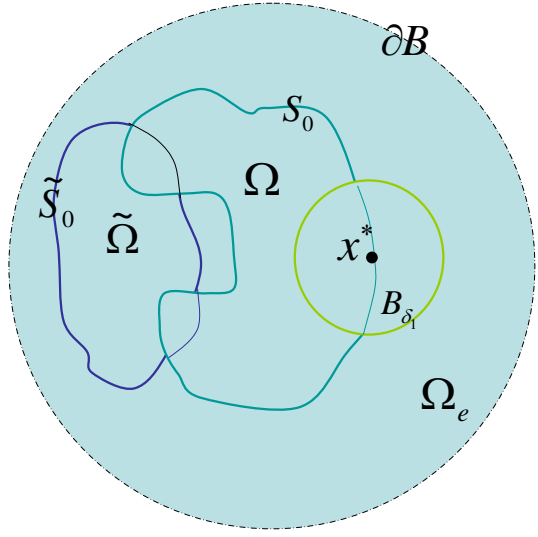}}
\caption{Two different interfaces}\label{fig2}
\end{figure}

\begin{lemma}\label{newEi}
If the electric far field patterns $E^{\infty}(\widehat{x},d,p)$ for
$S_{0}$ and $\wi{E}^{\infty}(\widehat{x},d,q)$ for $\wi{S}_{0}$
coincide for all $\widehat{x},d,q\in S^{2}$, then the electric far
field patterns also coincide for the incoming wave
 \be\label{Ei1}
 E^{i}(x)&=&\frac{i}{k_{0}}\curlcurl\int_{S_{0}}s(y)\Phi_{0}(x,y)ds(y),\cr
 H^{i}(x)&=&\curl\int_{S_{0}}s(y)\Phi_{0}(x,y)ds(y),\qquad x\in B\cap\Om_{e},
 \en
where $s\in T_{d}^{0,\al}(S_{0})$ has a compact support in
$S_{0}\cap B_{\delta_{1}}$.
\end{lemma}

\begin{proof}
By Rellich's Lemma \cite{CK}, the assumption
$E^{\infty}(\hat{x};d,q)=\wi{E}^{\infty}(\hat{x};d,q)$ for all
$\hat{x},d\in S^2,\;q\in \R^{3}$ implies that
$E^{s}(z,d,q)=\wi{E^{s}}(z,d,q)$ for all $z\in \Om_{e}$ and all
$d\in S^2,\;q\in \R^{3}$. Using the mixed reciprocity relation
(\ref{mixz0}), we obtain that the electric far field patterns
corresponding to both interfaces coincide for incoming waves of the form
 \be\label{Eied0}
 E^{i}(x;z,p)&=&\frac{i}{k_{0}}\curlcurl(p\Phi(x,z)),\cr
 H^{i}(x;z,p)&=&\curl(p\Phi_{0}(x,z)),\quad x\in \Om_{e},\;x\neq z.
 \en
Furthermore, $E^{i}$ and $H^{i}$ defined as in (\ref{Ei1}) are
continuous in $\ov{B\cap\Om_{e}}$ and can be uniformly approximated, respectively,
by fields of the form
 \be\label{Ei2}
 E^{i}_{\delta_{2}}(x)&=&\frac{i}{k_{0}}\curlcurl
    \int_{S_{0}}s(y)\Phi_{0}(x+\delta_{2}\nu(x^{\ast}),y)ds(y),\cr
 H^{i}_{\delta_{2}}(x)&=&\curl\int_{S_{0}}s(y)
    \Phi_{0}(x+\delta_{2}\nu(x^{\ast}),y)ds(y),\qquad x\in\ov{B\cap\Om_{e}}
 \en
by making $\delta_{2}$ sufficiently small. In fact, for $x\in
(\ov{B\cap\Om_{e}})\setminus B_{\delta_{1}}$ the differences between
$E^{i}_{\delta_{2}}(x)$ and $E^{i}(x)$ and between
$H^{i}_{\delta_{2}}(x)$ and $H^{i}(x)$ converge to zero as
$\delta_{2}\rightarrow 0$ since $B_{\delta_{1}}\cap
S_{0}-\delta_{2}\nu(x^{\ast})\subset \Om$ for all sufficiently small
$\delta_{2}>0$, so the kernels are smooth; for $x\in B_{\delta_{1}}$
the differences converge to zero since $E^{i}$ and
$H^{i}$ are uniformly continuous in $\ov{B\cap\Om_{e}}$ and
$B_{\delta_{1}}\cap \Om_{0}+\delta_{2}\nu(x^{\ast})\subset\Om_{e}$
for all sufficiently small $\delta_{2}>0$.

Finally, with the help of a quadrature formula, the integrals in the
definition of $E^{i}_{\delta_{2}}(x)$ and $H^{i}_{\delta_{2}}(x)$
can be uniformly approximated on $\ov{B\cap\Om_{e}}$ by sums of
waves of the form (\ref{Eied0}). Note that by the well-posedness of
the direct problem we conclude that the electric far field pattern
depends continuously on the incoming wave. Based on this result we
obtain the assertion of the lemma.
\end{proof}

We are now ready to prove the main result of this section.

\begin{lemma}\label{uniS0}
Let $\Im k_{1}>0$. If the electric far field patterns
$E^{\infty}(\widehat{x},d,q)$ for $S_{0}$ and
$\wi{E}^{\infty}(\widehat{x},d,q)$ for $\wi{S}_{0}$ coincide for all
observation directions $\widehat{x}\in S^{2}$, all incident
directions $d\in S^{2}$ and all polarizations $q\in \R^{3}$, then
$S_{0}=\wi{S}_{0}$.
\end{lemma}

\begin{proof}
We assume that $S_{0}\neq\wi{S}_{0}$. Then we can choose a point
$x^{\ast}$ as we did at the beginning of Section \ref{sec3}.
For the same $\delta_{1}$ as obtained there, we use a constant tangential vector
$p\neq0$ at $x^{\ast}$ and a smooth cut-off function $\rho$ that takes the value
one close to $x^{\ast}$ and zero for $x$ with $|x-x^{\ast}|\geq\delta_{1}/3$.
Define $x_{n}:=x^{\ast}+\delta_{1}/n\nu(x^{\ast})$, $n\in \N$. Denote by
$E_{n}, H_{n}, F_{n}, G_{n}$ the solution of the direct transmission
problem (\ref{0EH})-(\ref{0ibc}) for the incoming wave:
 \be\label{Ei0}
 E^{i}_{n}(x)&=&\la_{E}\curlcurl\int_{S_{0}}s_{n}(y)\Phi_{0}(x,y)ds(y),\cr
 H^{i}_{n}(x)&=&-ik_{0}\la_{E}\curl\int_{S_{0}}s_{n}(y)\Phi_{0}(x,y)ds(y),
   \qquad x\in B\cap\Om_{e},\;\;n\in\N,
 \en
where we set
 \ben
 s_{n}(x):=\frac{\rho(x)}{|x-x_{n}|^{1/2}}\nu(x)\times p,\qquad x\in S_{0},\;\;n\in \N.
 \enn
Then define
 \ben
 F^{\ast}_{n}(x)&:=&F_{n}(x)-\curlcurl\int_{S_{0}}s_{n}(y)\Phi_{1}(x,y)ds(y),\\
 G^{\ast}_{n}(x)&:=&G_{n}(x)+ik_{1}\curl\int_{S_{0}}s_{n}(y)\Phi_{1}(x,y)ds(y),
   \qquad x\in \Om_{1},\;\;n\in \N
 \enn
which satisfies the equations (\ref{0FG}). Since $s_{n}$ is
uniformly bounded in $T^{2}(S_{0})$ for all $n\in \N$, we have that
 \be
\label{T1ast} T^{\ast}_{1,n}(x)&:=&\nu\times E_{n}
             -\la_{E}\nu\times F^{\ast}_{n}=\la_{E}(N_{0,1}-N_{0,0})Rs_{n},\\
\label{T2ast} T^{\ast}_{2,n}(x)&:=&\nu\times
 H_{n}-\la_{E}\nu\times G^{\ast}_{n}=ik_{0}\la_{E}M_{0,0}s_{n}-ik_{1}\la_{H}M_{0,1}s_{n}
 \en
and $\Div T^{\ast}_{1,\;n}$ are uniformly bounded with respect to
the $L^{2}(S_{0})-$norm. Consider the integral equations (\ref{T1})-(\ref{divF=0})
in the space $T^{2}(S_{0})\times T^{2}(S_{0})\times T_{d}^{0,\al}(\G_{1})\times
T^{0,\al}(\G_{2})\times C^{0,\al}(\G_{2})$ with the right-hand sides
$T^{\ast}_{1,\;n}(x)$ and $T^{\ast}_{2,\;n}(x)$ defined as above and
 \ben
\label{T3ast} T^{\ast}_{3,n}(x)&:=&\nu\times F^{\ast}_{n}=-N_{1,1}Rs_{n},\qquad x\in \G_{1},\\
\label{T4ast} T^{\ast}_{4,n}(x)&:=&\nu\times G^{\ast}_{n}
          -\frac{\la}{k_{1}}(\nu\times F^{\ast}_{n})\times\nu
          =ik_{1}M_{2,1}s_{n}-\frac{\la}{k_{1}}RN_{2,1}Rs_{n},\quad x\in \G_{2}.
 \enn
It can be found that the densities $a_{n}^{\ast},b_{n}^{\ast},c_{n}^{\ast},d_{n}^{\ast}$
and $\psi_{n}^{\ast}$ are uniformly bounded with respect to the norm of the above space.
Using the regularity properties of the integral operators in (\ref{T1}) we
conclude that the tangential fields $a_{n}^{\ast}$ are even uniformly bounded
in $T_{d}^{2}(S_{0})$. Here, the condition $S_{0}\in C^{2,\al}$ is important since then
we are able to use the fact that $M_{0,0}$ and $M_{0,1}$ map $T^{2}(S_{0})$ continuously
into $T_{d}^{2}(S_{0})$ (see \cite{Kirsch89}). Let $B_{\delta_{1}/2}$ be
the ball centered at $x^{\ast}$ with radius $\delta_{1}/2$ and
define $\G^{\ast}:=S_{0}\cap B_{\delta_{1}/2}$ and $\G:=S_{0}\setminus B_{\delta_{1}/2}$.
Then we have that $\Div T^{\ast}_{2,\;n}$ is uniformly bounded in $L^{2}(\G)$.
Thus, it follows from (\ref{T2}) that $\Div b_{n}^{\ast}$ is uniformly bounded in $L^{2}(\G)$.
These results together with the solution representation (\ref{F}) imply that
$\nu\times F^{\ast}_{n}$ is uniformly bounded with respect to the $T_{d}^{2}(\G)-$norm.

By Lemma \ref{newEi} and Rellich's lemma $E_{n}$ and $H_{n}$ coincide with the scattered
fields from the structure with interface $\wi{S}_{0}$.
Since the incoming fields given by (\ref{Ei0}) are uniformly bounded on $\wi{S}_{0}$, $E_{n}$ and
$H_{n}$ together with their derivatives are uniformly bounded in
$B_{\delta_{1}/2}\subset \wi{\Om_{0}}$. Hence, by the equation (\ref{T1ast}) and the
regularity of $T^{\ast}_{1,\;n}$, $\nu\times F^{\ast}_{n}$ is also uniformly bounded
with respect to the $T_{d}^{2}(\G^{\ast})-$norm.

We now conclude by Lemma \ref{3.3} below that $\Div(\nu\times G^{\ast}_{n})$ remains
uniformly bounded in $L^{2}(S_{0})$ for all $n\in \N$.
This implies that $\Div [\la_{H}\nu\times G_{n}-\nu\times
H^{i}_{n}] =\la_{H}\Div (\nu\times G^{\ast}_{n})+\Div
[\la_{H}\nu\times (G_{n}-G^{\ast}_{n})-\nu\times H^{i}_{n}]$ is
uniformly bounded in $L^{2}(\G)-$norm since $s_{n}$ is compactly
supported in $B_{\delta_{1}/3}\cap S_{0}$ so the kernels
in $\la_{H}\nu\times (G_{n}-G^{\ast}_{n})-\nu\times H^{i}_{n}$ are smooth.
Furthermore, The function $\Div [\la_{H}\nu\times G_{n}-\nu\times H^{i}_{n}]
=\Div (\nu\times H_{n})$ is also uniformly bounded in $L^{2}(\G^{\ast})-$norm
since $H_{n}$ and its derivatives are uniformly bounded in $B_{\delta_{1}/2}$.
By the vector identity $\Div(\nu\times V)=-\nu\cdot\curl V$ we conclude
that the fields
 \ben
&&\Div[\la_{H}\nu\times (G_{n}-G^{\ast}_{n})-\nu\times H^{i}_{n}]\cr
&=&\Div[-ik_{1}\la_{H}\nu\times\curl\int_{S_{0}}s_{n}(y)\Phi_{1}(x,y)ds(y)\cr
 &&+ik_{0}\la_{E}\nu\times \curl\int_{S_{0}}s_{n}(y)\Phi_{0}(x,y)ds(y)]\cr
&=&-ik_{1}\la_{H}\nu\cdot\curlcurl\int_{S_{0}}s_{n}(y)\Phi_{1}(x,y)ds(y)\cr
 &&+ik_{0}\la_{E}\nu\cdot \curlcurl\int_{S_{0}}s_{n}(y)\Phi_{0}(x,y)ds(y)
 \enn
remains uniformly bounded in $L^{2}(S_{0})$. Thus, denoting by $\wi{\Phi}_{0}$
the fundamental solution of the Laplace equation, we have that
 \ben
&&\lim_{\eps\rightarrow0,\;\eps>0}\nu\cdot
\curlcurl\int_{S_{0}}s_{n}(y)\wi{\Phi}_{0}(x+\eps\nu(x),y)ds(y)\cr
&=&\lim_{\eps\rightarrow0,\;\eps>0}\frac{\pa}{\pa\nu(x)}{\rm div}
   \int_{S_{0}}s_{n}(y)\wi{\Phi}_{0}(x+\eps\nu(x),y)ds(y)
 \enn
is uniformly bounded in $L^{2}(S_{0})$. Here we have used the
identity $\curlcurl V=-\Delta V+\nabla{\rm div} V$ and the fact that
$k_{0}\la_{E}-k_{1}\la_{H}=k^{2}_{0}(1+k^{2}_{1}\la_{H}^{2})/({k_{1}\la_{H}})\neq0$
since $k_{0},\; \Re k_{1},\; \Im k_{1}$ and $\la_{H}$ are positive constants.
From this we conclude that
 \ben
 f_{n}(x)=\Div s_{n}(x)+2\int_{S_{0}}\Div s_{n}(y)
   \frac{\pa \Phi_{1}}{\pa\nu(x)}(x,y)ds(y),\qquad x\in S_{0}
 \enn
is uniformly bounded in $L^{2}(S_{0})$. Consequently, $\Div s_{n}$ remains uniformly
bounded in $L^{2}(S_{0})$ for all $n\in \N$ since the above integral equation is
uniquely solvable in $L^{2}(S_{0})$ and the inverse operator is continuous in $L^2(S_0)$.
Computing $\Div s_{n}$ and omitting the terms that are obviously uniformly bounded
in $L^{2}(S_{0})$ we have that
 \ben
 \frac{(\nu(x)\times p)\cdot(x-x^{\ast})}{|x-x^{\ast}|^{5/2}},\qquad x\in S_{0}
 \enn
is uniformly bounded in $L^{2}(S_{0})$. This is a contradiction as
can be seen by parameterizing $S_{0}$ locally around $x^{\ast}$.
This ends the proof.
\end{proof}

It remains to prove the following lemma.

\begin{lemma}\label{3.3}
Let $\Im k_{1}>0$. Assume that $T_{1,n}\in T_{d}^{0,\al}(S_{0})$ is uniformly bounded
in  $T_{d}^{2}(S_{0})$, $T_{3,n}$ is uniformly bounded in  $T_{d}^{0,\al}(\G_{1})$
and $T_{4,n}$ is uniformly bounded in  $T^{0,\al}(\G_{2})$ for all $n\in\N$.
Then there exists a unique solution $F^{\ast}_n,\; G^{\ast}_n\in C^1(\Om_1)\cap C(\ov{\Om_1})$
to the Maxwell equations
 \ben
 \curl F^{\ast}_n-ik_1G^{\ast}_n=0,\qquad\curl G^{\ast}_n+ik_1F^{\ast}_n=0\qquad {\rm in}\;\;\Om_1
 \enn
with boundary conditions
 \ben
 \nu\times F^{\ast}_{n}=T_{1,n}\qquad&& {\rm on}\;\; S_{0},\\
 \nu\times F^{\ast}_{n}=T_{3,n}\qquad&& {\rm on}\;\; \G_{1},\\
 \nu\times G^{\ast}_n-\frac{\la}{k_1}[\nu\times F^{\ast}_n]\times\nu=T_{4,n}\qquad&& {\rm on}\;\;\G_2
 \enn
for each $n\in \N$. Furthermore, $\Div (\nu\times G^{\ast}_{n})$ is
uniformly bounded in $L^{2}(S_{0})$ for all $n\in \N$.
\end{lemma}

\begin{proof}
We first prove the uniqueness result. Let $T_{1,n}=0$ on $S_{0}$, $T_{3,n}=0$ on $\G_{1}$
and $T_{4,n}=0$ on $\G_2$. Then we just need to prove that $F^{\ast}_n=G^{\ast}_n=0$ in $\Om_1$.
Using Green's vector theorem and the above Maxwell equations, we have
 \ben
0&=&k_{1}\int_{\Om_{1}}\{\curl G^{\ast}_{n}\cdot
\ov{F^{\ast}_{n}}+ik_{1}F^{\ast}_{n}\cdot \ov{F^{\ast}_{n}}\}dx\cr
&=&k_{1}\int_{\Om_{1}}\{G^{\ast}_{n}\cdot\curl\ov{F^{\ast}_{n}}+ik_{1}|F^{\ast}_{n}|^{2}\}dx
  -\int_{\G_{2}}\la |\nu\times F^{\ast}_{n}|^{2}ds\cr
&=&\int_{\Om_{1}}\{-i|k_{1}G^{\ast}_{n}|^{2}+i[(\Re k_{1})^{2}
  -(\Im k_{1})^{2}]|F^{\ast}_{n}|^{2}-2\Re k_{1}\Im k_{1}|F^{\ast}_{n}|^{2}\}dx
  -\int_{\G_{2}}\la |\nu\times F^{\ast}_{n}|^{2}ds
 \enn
Taking the real part of the above equation and noting that $\Im k_{1}>0$,
we get $F^{\ast}_{n}=G^{\ast}_{n}=0$ in $\Om_{1}$.

To prove the existence, we seek a solution in the form
 \be
\label{F1}F^{\ast}_{n}(x)&=&\curl\int_{S_{0}}a_{n}(y)\Phi_{1}(x,y)ds(y)\cr
 &&+\curl\int_{\G_{1}}c_{n}(y)\Phi_{1}(x,y)ds(y)
 +\frac{i}{k_{1}^{2}}\curlcurl\int_{\G_{1}}\nu(y)
 \times(\widehat{S}_{1}^{2}c_{n})(y)\Phi_{1}(x,y)ds(y)\cr
 &&+\int_{\G_{2}}d_{n}(y)\Phi_{1}(x,y)ds(y)
   +i\la\curl\int_{\G_{2}}\nu(y)\times(\widehat{S}_{2}^{2}d_{n})(y)\Phi_{1}(x,y)ds(y)\cr
 &&+{\rm grad}\int_{\G_{2}}\psi_{n}(y)\Phi_{1}(x,y)ds(y)
   +i\la\int_{\G_{2}}\nu(y)\psi_{n}(y)\Phi_{1}(x,y)ds(y),\\
\label{G1}G^{\ast}_{n}(x)&=&\frac{1}{ik_{1}}\curl F^{\ast}_{n}(x)
 \en
for $x\in \Om\setminus S_{1}$. The jump relations together with the regularity of the surface
potentials imply that $F^{\ast}_{n},\; G^{\ast}_{n}$ defined in (\ref{F1})-(\ref{G1}) solve
the mixed boundary problem in $\Om_1$ provided the densities
$a_{n},\;c_{n},\;d_{n},\;\psi_{n}$ satisfy
 \be
\label{T1n}-a_{n}+M_{0,1}a_{n}-\frac{1}{\la_{E}}N_{1}c_{n}-\frac{1}{\la_{E}}P_{1}d_{n}
  -\frac{1}{\la_{E}}Q_{1}\psi_{n}&=&2T_{1,n}\qquad {\rm on}\;\;S_{0}, \\
\label{T3n}c_{n}+L_{3}a_{n}+N_{3}c_{n}+P_{3}d_{n}
           +Q_{3}\psi_{n}&=&2T_{3,n}\qquad{\rm on}\;\;\G_{1},\\
\label{T4n}d_{n}+L_{4}a_{n}+N_{4}c_{n}+P_{4}d_{n}
           +Q_{4}\psi_{n}&=&2T_{4,n}\qquad{\rm on}\;\;\G_{2},\\
\label{divF=0n}i\la\psi_{n}+P_{5}d_{n}+Q_{5}\psi_{n}&=&0\qquad\qquad{\rm on}\;\;\G_{2}.
 \en
This system of integral equations is Fredholm in $Y:=T^{2}(S_{0})\times T_{d}^{0,\al}(\G_{1})\times
T^{0,\al}(\G_{2})\times C^{0,\al}(\G_{2})$. Therefore we just need
to prove that the corresponding homogeneous system has a trivial solution.
Similar to the argument in the proof of Theorem \ref{well-posedness}, we can
prove that $\psi_{n}=0$ on $\G_{2}$, $c_{n}=0$ on $\G_{1}$ and $d_{n}=0$ on $\G_{2}$.
Note that $F^{\ast}_{n}, G^{\ast}_{n}$ are also well-defined in $\Om_{0}$ and
are a radiation solution to the Maxwell equations
 \ben
 \curl F^{\ast}_n-ik_1G^{\ast}_n=0,\qquad\curl G^{\ast}_n+ik_1F^{\ast}_n=0\qquad\mbox{in}\;\;\Om_0
 \enn
with the homogenous boundary condition
 \ben
 \nu\times F^{\ast}_{n}=0 \qquad \mbox{on}\;\; S_{0}.
 \enn
Thus, by Theorem 6.18 in \cite{CK} we obtain that $F^{\ast}_{n}=0$ in $\Om_{0}$.
By the jump relations, we have
 \ben
 -a_{n}+M_{0,1}a_{n}=0=a_{n}+M_{0,1}a_{n}\qquad {\rm on}\;\; S_{0},
 \enn
so $a_{n}=0$ on $S_{0}$.

Since the right-hand side of the system of integral equations (\ref{T1n})-(\ref{divF=0n}) is bounded
in $Y$, then the solution $a,c,d,\psi$ is also bounded in $Y$. Thus, using the integral representation
(\ref{F1}) and by Lemma 1 in \cite{Ha},
$\Div (\nu\times G^{\ast}_{n})=-\nu\cdot \curl G^{\ast}_{n}=ik_{1}\nu\cdot F^{\ast}_{n}$
is uniformly bounded in $L^{2}(S_{0})$ for all $n\in\N$. This proves the lemma.
\end{proof}


\section{Unique determination of the boundary $S_1$ and its property}\label{sec4}
\setcounter{equation}{0}

In this section we will prove that, given that $S_{0}=\wi{S}_{0}$, the impenetrable obstacle
$\Om_{2}$ and its physical property $\mathscr{B}$ can be uniquely determined.
Note that in this problem we only assume that the wave number $k_{1}$ can be any constant
with $\Im k_{1}\geq0$, so the result obtained in this section is a generalization of
that in \cite{LZaa}.

\begin{lemma}\label{Esequal}
For $\Om_2,\wi{\Om}_2\subset\Om$, let $G$ be the unbounded component
of $\R^3\setminus(\ov{\Om_2\cup\wi{\Om}_2})$ and let
$E^{\infty}(\hat{x};d,q)=\wi{E}^{\infty}(\hat{x};d,q)$ for all
$\hat{x},d\in S^2,\;q\in \R^{3}$ with $\wi{E}^{\infty}(\hat{x};d,q)$
being the electric far field pattern of the scattered field
$\wi{E}^{s}(x;d,q)$ corresponding to the obstacle $\wi{\Om}_2$ and
the same incident plane wave $E^{i}(\hat{x};d,q)$. Let
$z\in\Om_1\cap G,$ $E^i=E^i(x;z,p),\;H^i=H^i(x;z,p)$ and let
$F=F(x;z,p)$ be the unique solution of the problem
 \be
 \label{4EH} \curl E-ik_0 H=0\;\;\curl H+ik_0 E=0\qquad&{\rm in}\;\;\Om_{0},\\
 \label{4FG} \curl F-ik_1 G=0\;\;\curl G+ik_1 F=0\qquad&{\rm in}\;\;\Om_{1},\\
 \label{4tbc}\nu\times E-\la_{E}\nu\times F=\la_{E}\nu\times E^i,\;\;
             \nu\times H-\la_{H}\nu\times G=\la_{H}\nu\times H^i \qquad&{\rm on}\;\;S_{0},\\
 \label{4pbc}\nu\times F=-\nu\times E^i\qquad&{\rm on}\;\;\G_{1},\\
 \label{4ibc}\nu\times G-\frac{\la}{k_{1}}(\nu\times F)\times\nu
     =-\nu\times H^i+\frac{\la}{k_{1}}(\nu\times E^i)\times\nu\qquad&{\rm on}\;\;\G_{2},\\
 \label{4rc} \lim_{|x|\rightarrow\infty}(H\times x-|x|E)=0.\qquad&
 \en
Let $\wi{F}$ =$\wi{F}(x;z,p)$ be the unique solution of the above
problem with $\Om_2$ replaced by $\wi{\Om_2}$ and $\Om_1$ replaced
by $\wi{\Om_1}:=\Om\ba(\ov{\Om_1})$. Then
 \be
 \label{eses} F(x;z,p)=\wi{F}(x;z,p),\qquad\forall\; x\in \Om_1\cap G.
 \en
\end{lemma}

\begin{remark}\label{r3.4} {\rm
By Theorem \ref{well-posedness}, the problem (\ref{4EH})-(\ref{4rc})
has a unique solution.
}
\end{remark}

\begin{proof}
By Rellich's Lemma \cite{CK}, the assumption that
$E^{\infty}(\hat{x};d,q)=\wi{E}^{\infty}(\hat{x};d,q)$ for all
$\hat{x},d\in S^2,\;q\in \R^{3}$ implies that
 \ben
  \nu\times E(x,d,q)=\nu\times\wi{E}(x,d,q),\qquad
  \nu\times H(x,d,q)=\nu\times\wi{H}(x,d,q)
 \enn
for all $x\in S_0$ and all $d\in S^2,\;q\in \R^{3}$.
By Holmgren's uniqueness theorem (see Lemma 3.2 in \cite{AN} or
Theorem 4.1.2.4 in \cite{Kress001}), it follows that
 \ben
 \wi{F}(z,d,q)=F(z,d,q),\qquad\forall \;d\in S^2,\;q\in \R^{3}.
 \enn
For the electric far field patterns corresponding to the electric dipole
we have by the mixed reciprocity relation (\ref{mixz1}) that
 \ben
E^\infty(d;z,p)=\wi{E}^\infty(d;z,p),\qquad\forall \;d\in S^2,\;q\in
\R^{3}.
 \enn
Rellich's Lemma \cite{CK} gives
 \ben
\nu\times E(x;z,p)=\nu\times \wi{E}(x;z,p),\quad
\nu\times H(x;z,p)=\nu\times \wi{H}(x;z,p)
 \enn
for all $x\in S_0,\;p\in \R^{3}$. By Holmgren's uniqueness theorem again
it is derived that
 \ben \wi{F}(x;z,p)=F(x;z,p),\qquad\forall\; x\in\Om_1\cap G,\;p\in\R^3,
 \enn
which is the desired result (\ref{eses}).
\end{proof}

Making use of Lemma \ref{Esequal} and the mixed reciprocity relation (\ref{mixz1})
we can prove the following uniqueness result.

\begin{theorem}\label{th4.1}
If there are two obstacles $\Om_{2}$ and $\wi{\Om}_{2}$ which lead
to the same far field pattern for all observation directions and all
incident directions at a fixed frequency, i.e.,
$E^{\infty}(\widehat{x},d,q)=\wi{E}^{\infty}(\widehat{x},d,q)$ for
all $\widehat{x},\;d\in S^2$ and all $q\in \R^{3}$, then
$S_{1}=\wi{S}_{1}$ and $\mathscr{B}=\wi{\mathscr{B}}$, that is, the
impenetrable obstacle with its physical property are uniquely determined.
\end{theorem}

\begin{proof}
Let G be the unbounded component of
$\R^3\setminus(\ov{\Om_2\cup\wi{\Om}_2})$. Assume that $S_{1}\neq\wi{S}_{1}$.
Then, without loss of generality, we may assume that there exists a point
$z_0\in S_1\cap\left(\R^3\ba\ov{\widetilde{\Om}_2}\right)$. We can choose
an $h>0$ such that the sequence
 \ben
z_{j}:=z_{0}+\frac{h}{j}\nu(z_{0}), \qquad j=1,2,\ldots,
 \enn
is contained in $G$, where $\nu(z_{0})$ is the outward normal to $S_{1}$ at $z_{0}$.
Consider the solution to the problem (\ref{4EH})-(\ref{4rc}) with $z$ being replaced
by $z_{j}$. By Lemma \ref{Esequal} it follows that $E^s(x;z_j,p)=\wi{E}^s(x;z_j,p)$ for all
$x\in\ov{G}$ and all polarizations $p\in\R^{3}$.
Since $z_0$ has a positive distance from $\ov{\wi{\Om}_2}$,
we conclude from the well-posedness of the direct problem that there exists a $C>0$ such
that $|\mathscr{B}(\wi{E}^s(z_0;z_j,p))|\leq C$ uniformly for $j\geq1$ and
all polarizations $p\in \R^{3}$. On the other hand, by the boundary condition on $S_1$,
 \ben
|\mathscr{B}(\wi{E}^s(z_0;z_j,p))|=|\mathscr{B}(E^s(z_0;z_j,p))|
  =|-\mathscr{B}(E^i(z_0,z_j,p))|\rightarrow\infty
 \enn
as $j\rightarrow\infty$ for $p\,\bot\,\nu(z_0)$. This is a
contradiction and therefore $S_{1}=\wi{S}_{1}$.

We now assume that the boundary conditions are different, that is,
$\mathscr{B}\neq\wi{\mathscr{B}}$. Since the obstacles and the far
field patterns are the same, we have $F(x,d,q)=\wi{F}(x,d,q)$.
Define $\Om_2:=\Om_2=\wi{\Om}_2$ and $F:=F(x,d,q)=\wi{F}(x,d,q)$ and
consider the case of impedance boundary conditions with two
different positive constants $\la\neq\wi{\la}$. Then from the
boundary conditions
 \be\label{twoipc}
 \nu\times G-\frac{\la}{k_{1}}(\nu\times F)\times\nu=0,\;\;
 \nu\times G-\frac{\wi{\la}}{k_{1}}(\nu\times F)\times\nu=0\qquad {\rm on}\;S_{1},
 \en
we observe that
 \ben (\la-\wi{\la})\nu\times F=0\qquad{\rm on}\;\;S_1.
 \enn
This, together with the boundary conditions (\ref{twoipc}), implies that
$\nu\times F=\nu\times G=0$ on $S_{1}$.
Therefore, by Holmgren's uniqueness theorem \cite{AN,Kress001}, $F=G=0$ in $\Om_{1}$.
Using Holmgren's uniquness theorem again and with the help of the transmission boundary
conditions, we conclude that $E(x,d,q)=E^{i}(x,d,q)+E^{s}(x,d,q)=0$ in $\Om_{0}$.
The scattered field $E^{s}(x,d,q)$ tends to zero uniformly at infinity,
while the incident plane wave $E^{i}(x,d,q)$ has modulus $|k_{0}(d\times q)\times d|$
everywhere. Thus the modulus of the total field tends to $|k_{0}(d\times q)\times d|$.
This leads to a contradiction, giving that $\la=\widetilde{\la}$.
The cases with other boundary conditions can be dealt with similarly.
\end{proof}

We summarize the main results for the inverse problem in the following theorem.

\begin{theorem}\label{th4.2}
Let $\Im k_{1}>0$. Then the interface $S_{0}$ and the obstacle
$\Om_{2}$ with its physical property $\mathscr{B}$ are uniquely
determined by the electric far field patterns
$E^{\infty}(\widehat{x},d,q)$ for all observation directions
$\widehat{x}\in S^2$, all incident directions $d\in S^2$ and all
polarizations $q\in \R^{3}$.
\end{theorem}

There is a widespread belief that the electric far field pattern for a single incident
direction $d\in S^{2}$ and a single polarization direction $q$ uniquely determines
the general obstacle, since the far field data depend on the same number of variables, as
does the obstacle to be recovered. However, this result remains a
challenging open problem \cite{CC05}. Recent progress has been made
by Liu, Zhang \& Zou \cite{LZZ09} who established uniqueness with a
single incident wave for a polyhedral obstacle and by Kress
\cite{Kress} who proved that a ball and its boundary condition (for
constant impedance $\la$) are uniquely determined by the far field
pattern for one incident plane wave. In a recent paper \cite{HLZ},
we proved uniqueness in determining a small perfectly conducting
ball in the inverse electromagnetic scattering problem by a finite
number of electric far field patterns with a single incident
direction and polarization. We now extend Kress's result to our case
that the background is a piecewise homogeneous medium.

\begin{corollary}\label{co4.1}
Let $\Im k_{1}>0$. Assume that the interface $S_{0}$ and the
boundary $S_{1}$ are concentric spheres with center at the origin
and the impedance $\la$ is a positive constant. Then they are
uniquely determined by the electric far field patterns
$E^{\infty}(\widehat{x},d,q)$ for all observation directions
$\widehat{x}\in S^2$, one fixed incident direction $d\in S^2$ and all
polarizations $q\in \R^{3}$.
\end{corollary}

\begin{proof}
By symmetry, the electric far field pattern for scattering of plane
waves by the concentric spheres $S_{0}$ and $S_{1}$ and the positive
constant impedance $\la$ satisfies $E^{\infty}(Q\widehat{x},Qd,Qq)=Q
E^{\infty}(\widehat{x},d,q)$ for all $\widehat{x},d\in S^2$, $q\in
\R^{3}$ and all rotations $Q$, that is, for all orthogonal
transformations with $\det Q=1$. Hence, knowledge of the electric
far field pattern for one incident direction implies knowledge of
the electric far field pattern for all incident directions. The
statement now follows from Theorem \ref{th4.2}.
\end{proof}

Karp's theorem in our case as stated in the following corollary is also
true.

\begin{corollary}\label{co4.2}
Let $\Im k_{1}>0$. Assume that the electric far field patterns
$E^{\infty}(\widehat{x},d,q)$ satisfies
$E^{\infty}(Q\widehat{x},Qd,Qq)=Q E^{\infty}(\widehat{x},d,q)$ for
all $\widehat{x},d\in S^2$, $q\in \R^{3}$ and all orthogonal
transformations $Q$ with $\det Q=1$. Then the interface $S_{0}$ and
the boundary $S_{1}$ are concentric spheres with center at the origin.
\end{corollary}

\begin{proof}
We define $\wi{S}_{0}=Q (S_{0})$ and $\wi{S}_{1}=Q (S_{1}) $
for some orthogonal transformation $Q$ with $\det Q=1$. Then, by
symmetry the corresponding electric far field pattern
$\wi{E}^{\infty}(\widehat{x},d,q)$ for $\wi{S}_{0},\;\wi{S}_{1}$ satisfies
 \ben
 \wi{E}^{\infty}(\widehat{x},d,q)=QE^{\infty}(Q^{-1}\widehat{x},Q^{-1}d,Q^{-1}q)
 =E^{\infty}(\widehat{x},d,q)
 \enn
for all $\widehat{x},d\in S^2$, $q\in \R^{3}$. Theorem \ref{th4.2} implies
that $\wi{S}_{0}=S_{0}$ and $\wi{S}_{1}=S_{1}$. This
holds for all orthogonal transformations $Q$. Therefore $S_{0}$ and
$S_{1}$ are concentric spheres with center at the origin.
\end{proof}

\section*{Acknowledgements}
This work was supported by the NNSF of China under grant No. 10671201.

\end{document}